\documentclass[10pt, twosided]{article}
\usepackage{latexsym,enumerate}
\usepackage{amsmath}
\usepackage{amsthm,amsopn}
\usepackage{amstext,amscd,amsfonts,amssymb,fontenc,bbm}
\usepackage{algpseudocode}
\usepackage{verbatim}
\usepackage{graphicx}
\usepackage{subfigure}
\usepackage{float}
\usepackage{algorithm}
\usepackage{enumitem} 

\usepackage{pstricks,pst-node}
\usepackage{tikz}
\usepackage{pgf,tikz}
   \usepackage{epsfig}

\usepackage{lineno} 
\usepackage{textcomp}
\usepackage{mathtools} 

\newtheorem{theorem}{Theorem}[section]
\newtheorem{corollary}[theorem]{Corollary}

\newtheorem{conjecture}[theorem]{Conjecture}
\newtheorem{proposition}[theorem]{Proposition}

\newtheorem{lemma}[theorem]{Lemma}

\newtheorem{remark}[theorem]{Remark}

\theoremstyle{definition}

\begin{document}

\title{An asymptotically tight upper bound for the domination number of the $2$-token graph of path graphs}

\author{E. Acosta Troncoso$^{*}$ \footnote{Unidad Acad\'emica de Matem\'aticas
Universidad Aut\'onoma de Zacatecas, Calzada Solidaridad y Paseo La Bufa
Col. Hidr\'aulica, C.P. 98060, Zacatecas, Zac.}\thanks{E-mail address: emmanuel\_acosta22@yahoo.com} \and J.~Lea\~nos$^{*}$ \thanks{E-mail address: jleanos@uaz.edu.mx}
\and L.~M.~Rivera$^{*}$ \thanks{E-mail address: luismanuel.rivera@gmail.com (corresponding author)}
}

\date{}
\maketitle

\begin{abstract}
 In this paper, we show that the domination number of the $2$-token graph of the path graph of order $n\geq 2$ 
is equal to $\frac{n^2}{10}+\Theta(n)$.
\end{abstract}


\section{Introduction}
Let $G=(V(G), E(G))$ be a simple graph of order $n\geq 2$. The {\em double vertex graph}  of $G$ is defined as the graph $F_2(G)$ whose vertex set consists of all $2$-subsets of $V(G)$, where two vertices are adjacent in $F_2(G)$ iff their symmetric difference is an edge of $G$. This concept, along with its generalization known as the {\em $k$-token graph} of $G$, has been independently redefined multiple times under various names. In 1989, the double vertex graph of $G$ was defined and extensively studied by Alavi et al.~\cite{alavi2, alavi3}; however, $F_2(G)$ can also be traced back to the PhD thesis of Johns~\cite{johns}, where $F_2(G)$ was called the {\em $2$-subgraph graph} of $G$.  In 2002, Rudolph~\cite{rudolph} redefined $F_2(G)$ as the {\em symmetric square power} of $G$ and employed it to address the graph isomorphism problem and various issues in quantum mechanics, which motivated numerous works by different authors (see, for example, \cite{alzaga, aude, barghi, fisch}). In 2012,  Fabila-Monroy et al.~\cite{FFHH} reintroduced this concept under the name of {\em $k$-token graph} of $G$, where $k\in \{1,2,\ldots, n\}$ and $F_k(G)$ is defined as the graph whose vertex set consists of all $k$-subsets of $V(G)$, where two vertices are adjacent in $F_k(G)$ iff their symmetric difference is an edge of $G$. In~\cite{FFHH} studied several combinatorial invariants of these graphs, such as connectivity, diameter, clique number, chromatic number, Hamiltonicity, etc. Following this approach, a wealth of results regarding various combinatorial parameters of token graphs has been published; see, for instance, \cite{dealba, alba,  token2, dalfo, FGGHT, gomez, leatrujillo}.

We recall that a {\em dominating set} of $G$ is a subset $S$ of vertices from $V(G)$ such that every vertex 
$v\in V(G)\setminus S$ is a neighbor of at least one vertex of $S$. The {\em domination number} $\gamma(G)$ of $G$ is defined as
\[
\gamma(G)=\min\{ |D| ~: D\; \text{ is a dominating set of G} \}.
\]

The study of dominating sets of graphs is a long-standing and highly active area of research in graph theory. Notably, various chess problems related to these sets were posed in India as early as the 17th century. One of the best-known examples is {\em The $n$-Queens problem}~\cite{vardi}, which seeks to determine the minimum number of queens required to dominate an $n \times n$ chessboard. In 1979, Garey and Johnson \cite{garey} proved that determining the domination number of a graph is an NP-hard problem. 

The aim of this paper is to establish an asymptotically tight upper bound for $\gamma(F_2(P_n))$, where $P_n$ denotes the path graph on $n \geq 2$ vertices. Recent related research includes the work of Goncalves et al.~\cite{goncalves}, who proved Chang's conjecture~\cite{chang} regarding the domination number of the {\em grid graph} $P_m \square P_n$, the Cartesian product of the paths $P_m$ and $P_n$. Specifically, they proved that for integers $m\geq n\geq 16$, 
\[
\gamma(P_m \square P_n)=\left\lfloor \frac{(n+2)(m+2)}{5}\right\rfloor-4.
\]
Additionally, in~\cite{prabhu}, general bounds for $\gamma(F_2(G))$ were presented. More recently, Fabila-Monroy and G\'omez-Galicia~\cite{sergio} determined the exact value of $\gamma(F_2(G))$ for 
$G \in \{K_n, K_{1,n-1}\}$. 


Our aim in this paper is to show the following result.
\begin{theorem}\label{teoprincipal}
If $n\geq 13$, then $a(n-1)\leq \gamma(F_2(P_n))\leq d(n)$, where
\[
a(n):=\begin{cases}
 \frac{1}{10}\left(n^2+n+20\right)  & n \equiv 0 \pmod 5\; \textup{or } n \equiv 4 \pmod 5,\\ 
             \frac{1}{10}\left(n^2+n+18\right)  & n\equiv 1 \pmod 5\; \textup{or } n \equiv 3 \pmod 5,\\
             \frac{1}{10}\left(n^2+n+14\right)  & n \equiv 2 \pmod  5,
\end{cases}
\]
and
\[
d(n):=\begin{cases}
 \frac{1}{10}(n^2+5n-30) & n \equiv 0 \pmod 5, \\
  \frac{1}{10}(n^2+5n-36)  & n \equiv 1 \pmod 5 \; \textup{or } n \equiv 4 \pmod 5,\\
\frac{1}{10}(n^2+5n-34) & n \equiv 2\pmod 5  \;  \textup{or } n \equiv 3 \pmod 5.
\end{cases}
\]
\end{theorem}

Although the proof of Theorem~\ref{teoprincipal} relies on a constructive approach that applies only for $n\geq 20$, we believe it is worthwhile, for the sake of completeness, to report the exact values of $\gamma(F_2(P_n))$ for $n\in\{2,3,\ldots,25\}$ in Table~1. These values were obtained computationally using a dynamic programming algorithm based on the framework for grid graphs developed by Alanko et al.~\cite{alanko}. We implemented a custom algorithm adapted to a subgraph $I(n)$ of $P_n \square P_n$ -- which will be formally defined in Section~\ref{prelimi} -- that is isomorphic to $F_2(P_n)$, by restricting the state space to exclude vertices outside the region defined by the boundary conditions of this subgraph. The transition rules were accordingly updated to handle these boundaries, effectively treating the empty region of the $n \times n$ grid as non-existent vertices. In Section~\ref{s:pseudocode}, we present a pseudocode description of our algorithms. We observe that $d(n) < 0$ for $n\in \{2,3\}$.  

\begin{table}[htp]
{\tiny
\begin{tabular}{|c|c|c|c|c|c |c |c |c |c |c |c |c |c |c |c |c |c|c|c |c |c |c |c|c|c |c |c|c|c | }
\hline
$n$ & 2& 3& 4 & 5 & 6 & 7 & 8 & 9 & 10 & 11 & 12& 13  & 14 & 15 & 16 & 17 & 18 & 19 & 20 & 21 & 22 & 23 & 24& 25\\ \hline
$\gamma(F_2(P_n))$ & 1& 1 & 2& 3& 5& 6& 8& 10& 12& 15& 18& 20 & 23 & 27& 30& 34& 38& 42& 47& 51& 56& 61& 66& 72\\ \hline
$d(n)$&-2 &-1 & 0& 2& 3& 5& 7& 9& 12& 14& 17& 20 & 23 & 27& 30& 34& 38& 42& 47& 51& 56& 61& 66 & 72\\ \hline
$a(n-1)$ & 2 & 2 & 3 & 4 & 5 & 6 & 7 & 9 & 11 & 13 & 15 & 17 & 20 & 23 & 26 & 29 & 32 & 36 & 40 & 44 & 48 & 52 & 57 & 62\\
\hline
\end{tabular} 
}

\caption{First values for $\gamma(F_2(P_n))$, $d(n)$, and $a(n-1)$.}
\label{tabladn}
\end{table}
As shown in Table~\ref{tabladn}, the gap between the upper and lower bounds $d(n)$ and $a(n-1)$ is relatively small. Specifically, we can express this difference as follows:
\begin{remark} If $n\geq 2$, then
\[
d(n) - a(n-1) = 
\begin{cases} 
\frac{1}{5}(3n - 25)&  n \equiv 0 \pmod{5}, \\
\frac{1}{5}(3n - 28) & n \equiv 1 \pmod{5}, \\
\frac{1}{5}(3n - 26) & n \equiv 2 \pmod{5}, \\
\frac{1}{5}(3n - 24) & n \equiv 3 \pmod{5}, \\
\frac{1}{5}(3n - 27) & n \equiv 4 \pmod{5}. 
\end{cases}
\]

\end{remark}
Based on the exact values presented in Table~\ref{tabladn}, where $\gamma(F_2(P_n)) = d(n)$ for $n \in \{13, 14, \ldots, 25\}$, we strongly believe the following:
\begin{conjecture}
If $n\geq 13$, then $\gamma(F_2(P_n))= d(n)$. 
\end{conjecture}

Recall that a set $S\subseteq V(G)$ is a  {\em packing set} of $G$ if every pair of distinct vertices $u, v\in S$ satisfy $d_G(u,v)\geq 3$, where $d_G(u,v)$ denotes the distance between $u$ and $v$ in $G$. The {\em packing number} $\rho(G)$ of $G$ is the maximum cardinality of a packing set of $G$. In 2018, G\'omez-Soto et al.~\cite{gomez} determined the exact value of $\rho(F_2(P_n))$. More precisely, they proved the following result.  

\begin{lemma}[{\rm\cite[Theorem 3]{gomez}}]\label{t:gomez}
If $n\geq 6$ and $a(n)$ is defined as in Theorem~\ref{teoprincipal}, then $\rho(F_2(P_{n+1}))=a(n)$.
\end{lemma}

\begin{remark} From Lemma~\ref{t:gomez} and the well-known fact that $\gamma(G)\geq \rho(G)$ for any graph $G$, it follows that $\gamma(F_2(P_n))\geq a(n-1)$. Therefore, to prove Theorem~\ref{teoprincipal}, it suffices to show that 
$\gamma(F_2(P_n))\leq d(n)$. 
\end{remark}


\section{Preliminaries}\label{prelimi}
We denote the set of integers by $\mathbb{Z}$ and the set of positive integers by $\mathbb{Z}^+$. Let $\mathbb{Z}^2=\mathbb{Z}\times \mathbb{Z}$. Let $H$ be the graph with vertex set $\mathbb{Z}^2$ in which two vertices $(i,i')$ and $(j,j')$ are adjacent  iff $|i-j|+|i'-j'|=1$.
For $n\in \{2, 3, \ldots\}$, let $H(n)$ be the subgraph of $H$ induced by the vertex set 
\[
\{(i,j) \mid 0\leq i, j\leq n \;\textup{ and } \; i-j\leq 1 \}\cup\{(n+1,n)\}. 
\]
Let $I(n)$ be the subgraph of $H(n)$ induced by the vertex set 
\[
\{(i,j) \mid 1\leq i, j \leq n-1  \;\textup{ and } \; i-j\leq 0 \}. 
\]

The left-hand side of Figure~\ref{f:H(10)&I(10)} shows the graphs $H(10)$ and $I(10)$.

\begin{figure}[!tbp]
  \centering
  \begin{minipage}[b]{0.4\textwidth}
    \includegraphics[width=\textwidth]{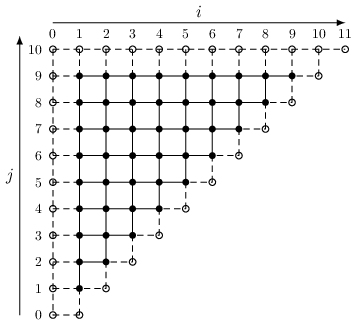}
  \end{minipage}
  \hfill
  \begin{minipage}[b]{0.4\textwidth}
    \includegraphics[width=\textwidth]{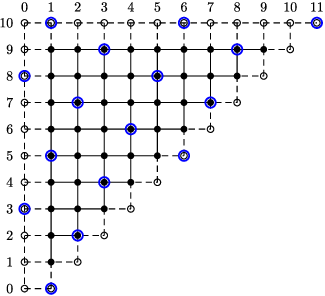}
  \end{minipage}
    \caption{On the left, $H(10)$ and $I(10)$; $I(10)$ is drawn with solid edges and black vertices. 
    On the right, the blue vertex set corresponds to the $f$-color $1$ (or equivalently, $f^{-1}(1)$).}
  \label{f:H(10)&I(10)}
\end{figure}

For every positive integer $n$, let $[n]:=\{1,2,\ldots,n\}$. If $0\leq m<n$ are integers, let $[m,n]:=\{m,m+1,\ldots,n\}$.
 We assume that $P_n$ has vertex set $[n]$ and edge set $E(P_n):=\{\{i,i+1\} \mid 1\leq i \leq n-1\}$. Then, 
 $V(F_2(P_{n}))=\{\{i,j\} \mid 1\leq i < j \leq n \}$. Let
\[
\phi: V(I(n)) \to V(F_2(P_n)), \quad (i, j) \mapsto \{i, j + 1\}.
\]

Under this function, we have $I(n) \simeq F_2(P_n)$.

Let $\mathbb{Z}_5:=\{0, 1, 2, 3, 4\}$ denote the set of congruence classes of integers modulo $5$. 
Let $f, g \colon \mathbb{Z}^2 \to \mathbb{Z}_5$ be the functions defined by
\[
f(x, y) = (x + 2y)\pmod{5} \quad \text{and} \quad g(x, y) = (2x + y)\pmod{5}.
\]
For $t\in \mathbb{Z}^+$ and $(x,y)\in \mathbb{Z}^2$, we define 
\[
f_t(x, y) = f(x - t, y) \quad \text{and} \quad  g_t(x, y) = g(x - t, y).
\]
We refer to  $f_t$ (respectively, $g_t$) as the \textit{t-translation} of $f$ (respectively, $g$). The elements of $\mathbb{Z}_5$ will be called the {\em colors} of $f$ (respectively, $g$), and for $l\in \mathbb{Z}_5$, we denote by  $f^{-1}(l)$ 
(respectively, $g^{-1}(l)$) the $l$-chromatic class of $H$ under $f$ (respectively, $g$). Figure~\ref{f:H(10)&I(10)}~(right) shows the $1$-chromatic class of $H(10)$ under $f$. 

The following properties of $f, g, f_t,$ and $g_t$ are largely immediate from their definitions; however, we provide a brief justification for the domination property in (2).
\begin{proposition}\label{propgf}
Let $i,i',j,j'\in \mathbb{Z}$ and $l\in \mathbb{Z}_5$.
\begin{enumerate}[label=\arabic*)]
\item[(1)] $f$ and $g$ define proper $5$-colorings of $H$.
\item[(2)] $f^{-1}(l)$ and $g^{-1}(l)$ are dominant sets of $H$.
\item[(3)] $f(i,j)=f(i,j')$ and $g(i,j)=g(i,j')$ iff $j-j'$ is a multiple of $5$.
\item[(4)] $f(i,j)=f(i',j)$ and $g(i,j)=g(i',j)$ iff $i-i'$ is a multiple of $5$.
\item[(5)] $f(i+1,i)=f(j+1,j)$ and $g(i+1,i)=g(j+1,j)$ iff  $i-j$ is a multiple of $5$.
\item[(6)] If $t \in \mathbb{Z}^+$, then $f^{-1}(f(i,j))=f_t^{-1}(f(i-t,j))$ and $g^{-1}(g(i,j))=g_t^{-1}(g(i-t,j))$.
\end{enumerate}
\end{proposition}
\begin{proof}
Properties (1) and (3)--(6) follow directly from the definitions. For (2), let $(i, j) \in V(H)$ be an arbitrary vertex satisfying 
$h(i, j) \neq l$, where $h\in \{f,g\}$. The set of neighbors of $(i, j)$ is $\{(i\pm 1, j), (i, j\pm 1)\}$. It can be verified that the $h$-colors of these four neighbors consist of exactly the four elements in $\mathbb{Z}_5 \setminus \{h(i, j)\}$. Since $l \in \mathbb{Z}_5 \setminus \{h(i, j)\}$, there exists at least one neighbor $(i', j')$ such that $h(i', j') = l$, which implies that $(i, j)$ is dominated by $h^{-1}(l)$.
\end{proof}
For $n\geq 4$, $r \in \mathbb{Z}_5$ and $h \in \{f, g\}$, let $h^{-1}_n(r):= h^{-1}(r) \cap V(H(n))$ denote the set of vertices of $H(n)$ with $h$-color $r$. We construct a dominating set $D_r(h)$ by augmenting the base set $I(n) \cap h^{-1}(r)$ with three auxiliary sets, $X_r^l(h)$, $X_r^u(h)$, and $X_r^d(h)$, which account for the left, upper, and diagonal boundary sides of $I(n)$, respectively. Note that $H(n)$ contains vertices lying just outside these boundaries, namely, those in the column $i=0$, the row $j=n$, and the sub-diagonal $j=i-1$. For each $w \in \{l,u,d\}$, the set $X_r^{w}(h)$ replaces those out-of-triangle vertices of $h^{-1}_n(r)$ with their nearest neighbors inside $I(n)$, thereby preserving the domination property within $V(I(n))$. These sets are formally defined as follows:

\begin{align*}\begin{split}
X_r^l(h)&=\{(1,i) \mid 0<i<n,\; (0,i)\in h^{-1}_n(r)\}\cup\{(1,1) \mid (0,0) \in  h^{-1}_n(r)\}\\
&\cup\{(1, n-1) \mid (0, n) \in  h^{-1}_n(r)\},\\
X_r^u(h)&=\{(i,n-1) \mid 0<i<n,\; (i,n)\in h^{-1}_n(r)\}\\
&\cup 
\{(n-1, n-1) \mid (n, n)  \text{ or } (n+1, n) \text{ belongs to } h^{-1}_n(r)\},\\
X_r^d(h)&=\{(i-1,i) \mid 0<i<n,\; (i,i-1)\in h^{-1}_n(r)\}\cup \{(1, 1) \mid (1, 0) \in h^{-1}_n(r)  \} \\
& \cup 
 \{(n-1,n-1) \mid (n,n-1)\in h^{-1}_n(r)\},\\
X_r(h)&= \left(X_r^l(h)\cap X_r^u(h)\right)\cup \left(X_r^l(h)\cap X_r^d(h)\right) \cup \left(X_r^u(h)\cap X_r^d(h)\right),\\  
D_r(h)&=\left(I(n)\cap h^{-1}(r)\right)\cup X_r^l(h)\cup X_r^u(h)\cup X_r^d(h).
\end{split}
\end{align*}

As an immediate consequence of Proposition~\ref{propgf}~(2) and the preceding definitions, we obtain the following corollary:
\begin{corollary}\label{coroD'} If $n, r, h$ and $D_r(h)$ are as above, then the following hold:
\begin{enumerate}[label=\arabic*)]
\item[(1)] $h^{-1}_n(r)$ is a dominating set for $I(n)$.
\item[(2)] $D_r(h)$ is a dominating set for $I(n)$ and $|D_r(h)|=|h^{-1}_n(r)|-|X_r(h)|$.
\end{enumerate}
\end{corollary}

As usual, for a set $U\subseteq V(I(n))$, its {\em closed neighborhood} is
\[
N[U] = U \cup \{ (x,y) \in V(I(n)) : (x,y) \text{ has a neighbor in } U \}.
\]

Although $D_r(h)$ is a dominating set of $I(n)$, it may contain redundant vertices. We therefore improve it by performing local exchanges: replacing a subset $A_r\subseteq D_r(h)$ with a smaller subset $B_r\subseteq V(I(n))\setminus D_r(h)$, provided that every vertex dominated exclusively by $A_r$ remains dominated by $B_r$.
This refinement process is already implicitly present in our definition of the auxiliary sets $X_r^l(h), X_r^u(h),$ and $X_r^d(h)$, which effectively perform such exchanges near the boundaries of $I(n)$ to maintain domination while reducing cardinality. The following proposition formalizes this exchange mechanism, and its proof is a straightforward exercise. 

\begin{proposition}\label{p:AB} Let $n, r, h$ and $D_r(h)$ be as above. If there exist $A_r\subset D_r(h)$ and 
$B_r\subset V(I(n))\setminus D_r(h)$ such that $N[D_r(h)]\subseteq N[(D_r(h)\setminus A_r)\cup B_r]$, then 
$D^*_r(h):=(D_r(h)\setminus A_r)\cup B_r$ is a dominating set of $I(n)$ with cardinality 
$|D^*_r(h)|=|D_r(h)|-|A_r|+|B_r|$. 
\end{proposition}
For integers $0\le k<r\le n$, let $H^{k,r}(n)$ denote the subgraph of $H(n)$
induced by the following subset of vertices: 
\[
  \bigl\{(x,y)\in V(H(n)) \;:\; k\le y\le r\bigr\}.
\]
We remark that $H^{k,r}(n)$ contains every vertex of $H(n)$ whose second
coordinate $y$ satisfies $k\le y\le r$. In particular, when $r=n$, it contains the vertex $(n+1,n)$.

\begin{proposition}\label{conteoh}
Let $n\geq i \geq 4$, $h\in \{g,f\}$ and $l \in \mathbb{Z}_5$. Then $|H^{i-4,i}(n)\cap h^{-1}(l)|=i$.
\end{proposition}
\begin{proof} We will only prove the case $h=g$, as the case $h=f$ follows similarly. First, we show that  $|H^{i-4,i}(n)\cap g^{-1}(r)|=|H^{i-4,i}(n)\cap g^{-1}(r')|$ for any $r,r' \in [0,4]$. Since the colors $c_0:=g(0,l)$, $c_1:=g(1,l)$, $c_2:=g(2,l)$, $c_3:=g(3,l)$, $c_4:=g(4,l)$ are pairwise distinct by Proposition~\ref{propgf}~(4), it is enough to show that $|H^{i-4,i}(n)\cap g^{-1}(c_t)|=|H^{i-4,i}(n)\cap g^{-1}(c_{t+1})|$ for any $t\in [0,3]$. Without loss of generality, we assume that $c_t$ is the color blue and that $c_{t+1}$ is the color red. Let $b:=|H^{i-4,i}(n)\cap g^{-1}(c_t)|$. According to the definition of $g_1$, the vertices of $H$ colored red by $g$ are precisely the vertices colored blue by $g_1$.  Let $b_1:= |H^{i-4,i}(n)\cap g_1^{-1}(c_t)|$ (we are counting the red vertices by means of the blue ones; note that the blue vertices are consistently positioned to the left of the red ones in the grid structure). Then Proposition~\ref{propgf}~(6) implies that $b_1=b+p-q$, where $p$ is the number of vertices in 
\[
\{(-1,i-4), (-1,i-3),(-1,i-2),(-1,i-1),(-1,i) \}
\] that are colored blue by $g$, and $q$ is the number of vertices in 
\[
\{(i-3,i-4),(i-2,i-3),(i-1,i-2),(i,i-1), (i+1,i)\}
\]
that are colored blue by $g$. Note that while the set $\{(-1, i-4), \dots, (-1, i)\}$ lies outside of $H(n)$, these vertices are considered as auxiliary points in the infinite grid $\mathbb{Z}^2$ to evaluate the color distribution relative to the shift induced by $g_1$. By Proposition \ref{propgf}~(3) and (5), it follows that $p=q=1$. Indeed, since the color function $g(x,y) = 2x+y \pmod 5$ defines a periodic tiling of $\mathbb{Z}^2$ with 5 color classes, any set of 5 consecutive vertices along a column, row, or diagonal
 that crosses the boundary of $H(n)$ necessarily contains exactly one vertex from each color class. Thus, $p=q=1$ is a direct consequence of this periodic structure. Therefore, we conclude that $b=b_1$, which implies that $|H^{i-4,i}(n)\cap g^{-1}(c_t)|=|H^{i-4,i}(n)\cap g^{-1}(c_{t+1})|$. Since this holds for any $t \in [0,3]$, we have $|H^{i-4,i}(n)\cap g^{-1}(r)|=|H^{i-4,i}(n)\cap g^{-1}(r')|$ for any $r,r' \in [0,4]$. To determine this common value, we count the vertices of $H^{i-4,i}(n)$
row by row. For $0\le j\le n-1$, the row $j$ of $H(n)$ consists of the vertices
$(x,j)$ with $0\le x\le j+1$, and hence has $j+2$ vertices. The row $j=n$
consists of the $n+1$ vertices $(x,n)$, where $0\le x\le n$, together with the
vertex $(n+1,n)$, and hence it has $n+2$ vertices as well. Therefore, in every case
the five rows $j=i-4,\dots,i$ of $H^{i-4,i}(n)$ contain $i-2$, $i-1$, $i$, $i+1$
and $i+2$ vertices, respectively. Thus, the total number of vertices is $(i-2)+(i-1)+i+(i+1)+(i+2)=5i$. Since the vertices are equally distributed among the five color classes, we conclude that $|H^{i-4,i}(n)\cap g^{-1}(r)| = 5i/5 = i$ for any $r \in [0,4]$.
\end{proof}



\section{$\gamma(F_2(P_n))\leq d(n)$: Proof of Theorem~\ref{teoprincipal}}

The exact value of $\gamma(F_2(P_n))$ for $n\leq 25$ was obtained by computer. See Section~\ref{s:pseudocode}. 
For $n\geq 20$, we employ a constructive approach. Our strategy to show that $\gamma(F_2(P_n))\leq d(n)$ is as follows:

{\bf Step 1}. Based on the congruence of $n$ modulo $5$, we select $h\in \{f,g\}$ and consider the set $h_n^{-1}(r)$,
where $n \equiv r \pmod 5$. The vertices of $h_n^{-1}(r)$ are enclosed by blue circles in Figures~2, 3, 4, 5 and 6.  

{\bf Step 2}. From $h_n^{-1}(r)$ we obtain $D_r(h)$, which is a dominating set of $I(n)$ 
with cardinality $|h^{-1}_n(r)|-|X_r(h)|$.  In Figures~2, 3, 4, 5 and 6, arrows indicate the transformation from $h_n^{-1}(r)$ to $D_r(h)$ by swapping the set of tail-vertices with the set of head-vertices. 

{\bf Step 3}. We then select appropiate small vertex subsets 
$A_r$ and $B_r$ from $I(n)$ that satisfy the following conditions: $A_r\subset D_r(h)$, $B_r\subset V(I(n))\setminus D_r(h)$, $|A_r|>|B_r|$, each vertex of $A_r\cup B_r$ is near some corner of $I(n)$, and $D^*_r(h):=(D_r(h)\setminus A_r)\cup B_r$ is a dominating set of $I(n)$. In Figures~2, 3, 4, 5 and 6,  the vertices of $A_r$ are enclosed in red squares, while those of $B_r$ are enclosed in green squares. 
 
{\bf Step 4}. Finally, we confirm that $D^*_r(h)$ is the required dominating set by showing that $|D^*_r(h)|=d(n)$. 

We emphasize that the three drawings in each of Figures~2, 3, 4, 5, and 6, correspond to the three big corners of $I(n)$ and illustrate the transformations described above. In the following cases ($i = 0, \dots, 4$), we present the sets $A_i$ and $B_i$, together with their closed neighborhoods. In each instance, the elements of $A_i$ and $B_i$ are highlighted in bold within their respective neighborhood sets.
\vskip 0.3cm
\noindent\textbf{Case 0}. Suppose that $n \equiv 0 \pmod 5$. We take $h = f$ and $r = 4$. A direct computation using 
Proposition~\ref{propgf}~(3)--(5) gives:
\begin{align*}
X_4^l(f) &= \{(1, 2), (1, 7), (1, 12), \dots, (1, n - 3)\}, \\
X_4^u(f) &= \{(4, n - 1), (9, n - 1), (14, n - 1), \dots, (n - 1, n - 1)\}, \\
X_4^d(f) &= \{(1, 2), (6, 7), (11, 12), \dots, (n - 4, n - 3)\},
\end{align*}

where $X_4^l(f)$ and $X_4^u(f)$ consist of vertices of the forms $(1, j)$ with $j < n - 1$ and $(i, n - 1)$ with $i > 1$, respectively, and $X_4^d(f)$ consists of vertices of the form $(i - 1, i)$. Consecuently, the only non-empty pairwise intersection is $X_4^l(f) \cap X_4^d(f) = \{(1, 2)\}$, giving $X_4(f) = \{(1, 2)\}$.

From Corollary~\ref{coroD'}~(2), $D_4(f)$ is a dominating set of $I(n)$ with cardinality $|D_4(f)| = |f_n^{-1}(4)| - 1$. By Proposition 2.4,
\[
 |D_4(f)| = -1+\sum_{i=1}^{n/5}\left|H^{5i-4,5i}(n)\cap f^{-1}(4)\right|=-1 + \sum_{i=1}^{n/5} 5i = \frac{n(n + 5)}{10} - 1.
 \]
Let the exchange sets near the corners (see Figure~\ref{fcason0}) be:
\begin{align*}
A_0 &:= \{(1, n - 1), (4, n - 1), (n - 6, n - 1), (n - 4, n - 1), (n - 1, n - 1), (1, n - 3)\}, \\
B_0 &:= \{(3, n - 1), (n - 5, n - 1), (n - 2, n - 1), (1, n - 2)\}.
\end{align*}

Set $D^*_4(f) := (D_4(f) \setminus A_0) \cup B_0$. Since $D_4(f)$ is a dominating set of $I(n)$, it suffices to verify that $N[A_0]$ is still dominated in $D^*_4(f)$. A direct computation gives:

\begin{align*}
N[A_0] = \{ & \mathbf{(1, n - 1)}, (2, n - 1), (1, n - 2), 
 \mathbf{(4, n - 1)}, (3, n - 1), (5, n - 1), (4, n - 2), 
 \mathbf{(n - 6, n - 1)},\\ 
 & (n - 7, n - 1), (n - 5, n - 1), (n - 6, n - 2),  \mathbf{(n - 4, n - 1)}, (n - 3, n - 1), (n - 4, n - 2), \\
& \mathbf{(n - 1, n - 1)}, (n - 2, n - 1), 
 \mathbf{(1, n - 3)}, (1, n - 4), (2, n - 3) \}.
\end{align*}
whereas
\begin{align*} 
N[B_0]=\{ & \mathbf{(3, n - 1)}, (2, n-1), (3, n-2), (4, n-1), \mathbf{(n - 5, n - 1)}, (n-6, n-1),  (n-4, n-1),\\
& (n-5, n-2), \mathbf{(n - 2, n - 1)},  (n - 3, n - 1),  (n - 1, n - 1),  (n - 2, n - 2), \mathbf{(1, n - 2)}, \\
& (1, n - 1), (1, n - 3), (2, n - 2)\}.
\end{align*}

Since any vertex in $N[A_0]$ adjacent to an element of $B_0$ is dominated by $B_0 \subseteq D^*_4(f)$, it remains to verify that $B_0$ dominates every vertex in $S_0:= N[A_0] \setminus N[B_0]$. It is easy to see that  

\[S_0=\{(5, n - 1), (4, n - 2), (1, n - 4), (2, n - 3),  (n - 7, n - 1), (n - 6,n - 2), (n - 4, n - 2)\}. \]

The following table exhibits explicit dominators in $D^*_4(f) \setminus A_0$ for each vertex in $S_0$.
\begin{center}
\begin{tabular}{|c|c|}
\hline
$(x,y)\in D^*_4(f) \setminus A_0$ & Vertices of $S_0$ dominanted by $(x,y)$  \\
\hline
$(6,n-1)$& $(5,n-1)$ \\
\hline
$(3,n-2)$ & $(4,n-2)$ \\
\hline
$(2,n-4)$&$ (1, n-4), (2,n-3)$ \\
\hline
$(n-7,n-2)$ & $(n-7,n-1), (n-6, n-2)$\\
\hline
$(n-4,n-3)$ & $(n-4,n-2)$\\
\hline
\end{tabular}
\end{center}
Thus, $D^*_4(f)$ is a dominating set of $I(n)$ with cardinality:
\[ |D^*_4(f)| = \frac{n(n + 5)}{10} - 1 - |A_0| + |B_0| = \frac{n(n + 5)}{10} -1- 6 + 4 = \frac{n^2 + 5n - 30}{10}. 
\]

\begin{figure}[H]
\begin{center}
\includegraphics[scale=1.33]{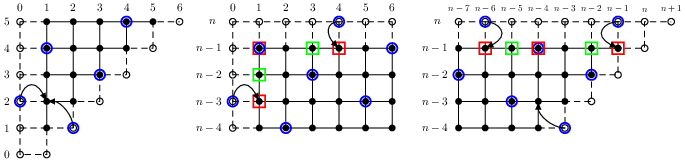}
\caption{The vertices of $A_0$ are enclosed in red squares, while those of $B_0$ are enclosed in green squares. We note that  $N[D_4(f)]\subseteq N[(D_4(f)\setminus A_0)\cup B_0]$.}
\label{fcason0}
\end{center}
\end{figure}

\vskip 0.3cm
\noindent\textbf{Case 1}. Suppose that $n \equiv 1 \pmod 5$. We take $h = g$ and $r = 4$. A direct computation using Proposition~\ref{propgf}~(3)--(5) gives:
\begin{align*}
X_4^l(g) &= \{(1, 4), (1, 9), (1, 14), \dots, (1, n - 2)\},\\
X_4^u(g) &= \{(4, n - 1), (9, n - 1), (14, n - 1), \dots, (n - 2, n - 1)\},\\
X_4^d(g) &= \{(4, 5), (9, 10), (14, 15), \dots, (n - 2, n - 1)\}.
\end{align*}
We note that the only non-empty pairwise intersection is $X_4^u(g) \cap X_4^d(g) = \{(n - 2, n - 1)\}$, which implies 
$X_4(g) = \{(n - 2, n - 1)\}$.

From Corollary~\ref{coroD'}~(2), $D_4(g)$ is a dominating set of $I(n)$ with cardinality $|D_4(g)| = |g_n^{-1}(4)| - 1$. Since $H^{0,1}(n) \cap g^{-1}(4) = \emptyset$, from Proposition 2.4 we obtain:
\[
|D_4(g)|=-1+\sum_{i=1}^{(n-1)/5}\left|H^{(5i+1)-4,5i+1}(n)\cap g^{-1}(4)\right|=-1+\sum_{i=1}^{(n-1)/5}(5i+1)= \frac{n^2 + 5n - 6}{10} - 1.
\]
Let the exchange sets near the corners (see Figure~\ref{fcason1}) be:
\begin{align*}
A_1 &:= \{(1, 9), (1, 7), (2, 5), (2, n - 1), (4, n - 1)\}, \\
B_1 &:= \{(1, 8), (2, 6), (3, n - 1)\}.
\end{align*}

Set $D^*_4(g) := (D_4(g) \setminus A_1) \cup B_1$. Since $D_4(g)$ is a dominating set of $I(n)$, it suffices to verify that the closed neighborhood $N[A_1]$ remains dominated in $D^*_4(g)$. Based on the structure of $I(n)$, an explicit calculation yields:

\begin{align*}
N[A_1] = \{ & \mathbf{(1, 9)}, (1, 8), (1, 10), (2, 9), 
 \mathbf{(1, 7)}, (1, 6), (2, 7), \mathbf{(2, 5)}, (2, 4), (2, 6), (3, 5), (1, 5), \\
&  \mathbf{(2, n - 1)}, (1, n - 1), (3, n - 1), (2, n - 2),  \mathbf{(4, n - 1)}, (4, n - 2), (5, n - 1) \}.
\end{align*}

whereas

\begin{align*}
N[B_1] = \{ & \mathbf{(1, 8)}, (1, 9), (1, 7), (2, 8), (1, 6), \mathbf{(2, 6)}, (3, 6), (2, 7), (2, 5),  \mathbf{(3, n - 1)}, (2, n - 1),\\
& (4, n - 1), (3, n - 2) \}.
\end{align*}

It remains to verify that $B_1$ dominates every vertex in $S_1:= N[A_1] \setminus N[B_1]$. It is easy to see that 
\[ S_1= \{(1, 10), (2, 9), (1, 5), (2, 4), (3, 5), (1, n - 1), (2, n - 2), (4, n - 2), (5, n - 1)\}. \]

The following table exhibits explicit dominators in $D^*_4(g) \setminus A_1$ for each vertex in $S_1$.
\begin{center}
\begin{tabular}{|c|c|}
\hline
$(x,y)\in D^*_4(g) \setminus A_1$ & Vertices of $S_1$ dominanted by $(x,y)$  \\
\hline
$(2,10)$& $(1,10), (2,9)$ \\
\hline
$(1,4) $& $(1,5), (2,4)$ \\
\hline
$(4,5)$&$ (3,5)$ \\
\hline
$(1,n-2)$ & $(1,n-1), (2,n-2)$\\
\hline
$(5,n-2)$ & $(5,n-1), (4,n-2)$\\
\hline
\end{tabular}
\end{center}

Therefore, $D^*_4(g)$ is a dominating set of $I(n)$ with cardinality:
\[
 |D^*_4(g)| = \frac{n^2 + 5n - 6}{10} - 1 - |A_1| + |B_1| = \frac{n^2 + 5n - 6}{10} - 6 + 3 = \frac{n^2 + 5n - 36}{10}.
 \]

\begin{figure}[H]
\begin{center}
\includegraphics[scale=1.3]{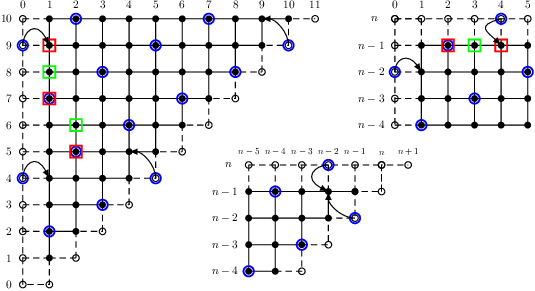}
\caption{The vertices of $A_1$ are enclosed in red squares, while those of $B_1$ are enclosed in green squares. We note that  $N[D_4(g)]\subseteq N[(D_4(g)\setminus A_1)\cup B_1]$.}
\label{fcason1}
\end{center}
\end{figure}

\vskip 0.3cm
\noindent\textbf{Case 2}. Suppose that $n \equiv 2 \pmod 5$. We take $h = g$ and $r = 0$. 
Combining Proposition~\ref{propgf}~(3)--(5) with the shift conditions of the auxiliary boundary sets, we obtain the following:
\begin{align*}
X_0^l(g) &= \{(1, 1), (1, 5), (1, 10), \dots, (1, n - 2)\}, \\
X_0^u(g) &= \{(4, n - 1), (9, n - 1), (14, n - 1), \dots, (n - 3, n - 1)\}, \\
X_0^d(g) &= \{(1, 2), (6, 7), (11, 12), \dots, (n - 6, n - 5)\}.
\end{align*}

We note that $X_0^l(g),~X_0^u(g)$, and $X_0^d(g)$ are pairwise disjoint. Therefore, $X_0(g)=\emptyset.$

From Corollary~\ref{coroD'}~(2), $D_0(g)$ is a dominating set of $I(n)$ with cardinality $|D_0(g)| = |g_n^{-1}(0)|$. Since $H^{0,2}(n)\cap g^{-1}(0)=\{(0,0), (2,1)\}$, then Proposition~\ref{conteoh} implies
\[
|D_0(g)|=2+\sum_{i=1}^{(n-2)/5}\left|H^{(5i+2)-4,5i+2}(n)\cap g^{-1}(0)\right|=2+\sum_{i=1}^{(n-2)/5}(5i+2)=\frac{n^2 + 5n + 6}{10}.
\]
Let the exchange sets near the corners (see Figure~\ref{fcason2}) be:
\begin{align*}
A_2 &:= \{ (1, 1), (1, 3), (1, 5), (1, 8), (1, 13), (1, 15), (2, 6), (2, 11), (3, 9), \\
    &\phantom{:= \{} (4, 7), (6, 8), (2, n - 1), (4, n - 1), (n - 3, n - 1), (n - 1, n - 1) \}, \\
B_2 &:= \{ (1, 6), (1, 14), (2, 4), (2, 8), (2, 12), (3, 6), (3, 10), (4, 8), (7, 9), \\
    &\phantom{:= \{} (3, n - 1), (n - 2, n - 1) \}.
\end{align*}

Set $D^*_0(g) := (D_0(g) \setminus A_2) \cup B_2$. Since $D_0(g)$ is a dominating set of $I(n)$, it suffices to verify that the closed neighborhood $N[A_2]$ remains dominated in $D^*_0(g)$. Based on the structure of $I(n)$, an explicit calculation yields the following sets:
\begin{align*}
N[A_2] = \{ 
  & \mathbf{(1, 1)}, (1, 2), \mathbf{(1, 3)}, (1, 4), (2, 3), \mathbf{(1, 5)}, (1, 6), (2, 5),  \mathbf{(1, 8)}, (1, 7), (1, 9), (2, 8), \\
  & \mathbf{(1, 13)}, (1, 12), (1, 14), (2, 13),   \mathbf{(1, 15)}, (1, 16), (2, 15),  \mathbf{(2, 6)}, (3, 6), (2, 7), \mathbf{(2, 11)},  \\
  &  (1, 11), (3, 11), (2, 12), (2, 10), 
  \mathbf{(3, 9)}, (2, 9), (3, 8), (3, 10), (4, 9), \mathbf{(4, 7)}, (3, 7), (4, 6), \\
  &  (4, 8), (5, 7), \mathbf{(6, 8)}, (5, 8), (6, 7), (6, 9), (7, 8), \mathbf{(2, n - 1)}, (1, n - 1), (3, n - 1), \\
  & (2, n - 2),  \mathbf{(4, n - 1)},  (5, n - 1), (4, n - 2), \mathbf{(n - 3, n - 1)}, (n - 4, n - 1), \\
  &  (n - 3, n - 2), (n - 2, n - 1),  \mathbf{(n - 1, n - 1)}\},
\end{align*}
whereas
\begin{align*}
N[B_2] = \{ 
  & \mathbf{(1, 6)}, (1, 5), (1, 7), (2, 6), 
   \mathbf{(1, 14)}, (1, 13), (1, 15), (2, 14), 
   \mathbf{(2, 4)}, (1, 4), (3, 4), (2, 3), \\
   & (2, 5),  \mathbf{(2, 8)}, (1, 8), (3, 8), (2, 7), (2, 9), 
 \mathbf{(2, 12)}, (1, 12), (3, 12), (2, 11), (2, 13),  \mathbf{(3, 6)}, \\
 &(4, 6), (3, 5), (3, 7),  \mathbf{(3, 10)}, (2, 10), (4, 10), (3, 9), (3, 11),
  \mathbf{(4, 8)}, (5, 8), (4, 7), (4, 9), \\
  & \mathbf{(7, 9)}, (6, 9), (8, 9), (7, 8), (7, 10), 
\mathbf{(3, n - 1)}, (2, n - 1), (4, n - 1), (3, n - 2), \\
 &   \mathbf{(n - 2, n - 1)}, (n - 3, n - 1), (n - 1, n - 1), (n - 2, n - 2) \}.
\end{align*}

It remains to verify that $B_2$ dominates every vertex in $S_2:= N[A_2] \setminus N[B_2]$. 
It is easy to see that
\begin{align*}
S_2= \{ & (1, 1), (1, 2), (1, 3), (1, 9), (1, 11), (1, 16), (2, 15), (5, 7),  (6, 7), (6, 8),  (1, n - 1),\\ 
&  (2, n - 2), (4, n - 2), (5, n - 1), (n - 3, n - 2),  (n - 4, n - 1) \}. 
\end{align*}

The following table exhibits explicit dominators in $D^*_0(g) \setminus A_2$ for each vertex in $S_2$. 
\begin{center}
\begin{tabular}{|c|c|}
\hline
$(x,y)\in D^*_0(g) \setminus A_2$ & Vertices of $S_2$ dominanted by $(x,y)$  \\
\hline
$(1,2)$& $(1,1), (1,2), (1,3)$ \\
\hline
$(1,10) $& $(1,9), (1,11)$ \\
\hline
$(2,16)$&$ (1,16), (2,15)$ \\
\hline
$(6,7)$ & $(5,7), (6,7), (6,8)$\\
\hline
$(1,n-2)$ & $(1,n-1), (2,n-2)$\\
\hline
$(5,n-2)$ & $(4,n-2), (5,n-1)$\\
\hline
$(n-2,n-2)$ & $(n-3,n-2)$\\
\hline
$(n-4,n-3)$ & $(n-4,n-1)$\\
\hline
\end{tabular}
\end{center}

Therefore, $D^*_0(g)$ is a dominating set of $I(n)$. Since $|A_2| = 15$ and $|B_2| = 11$, then
\[
 |D^*_0(g)| = \frac{n^2 + 5n + 6}{10} - |A_2| + |B_2| = \frac{n^2 + 5n + 6}{10} - 15 + 11 = \frac{n^2 + 5n - 34}{10}. 
\]

\begin{figure}[H]
\begin{center}
\includegraphics[scale=1.1]{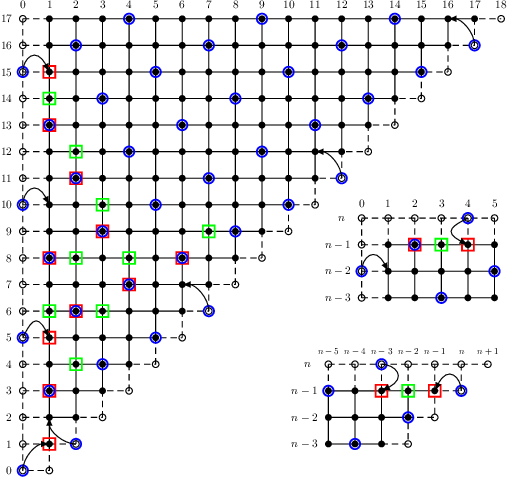}\\
\caption{The vertices of $A_2$ are enclosed in red squares, while those of $B_2$ are enclosed in green squares. We note that  $N[D_0(g)]\subseteq N[(D_0(g)\setminus A_2)\cup B_2]$.}
\label{fcason2}
\end{center}
\end{figure}

\vskip 0.3cm
\noindent\textbf{Case 3}. Suppose that $n \equiv 3 \pmod 5$. We take $h = g$ and $r = 4$. Combining Proposition~\ref{propgf}~(3)--(5) with the shift conditions of the auxiliary boundary sets, we obtain the following:
\begin{align*}
X_4^l(g) &= \{(1, 4), (1, 9), (1, 14), \dots, (1, n - 4)\}, \\
X_4^u(g) &= \{(3, n - 1), (8, n - 1), (13, n - 1), \dots, (n - 5, n - 1)\} \cup \{(n - 1, n - 1)\}, \\
X_4^d(g) &= \{(4, 5), (9, 10), (14, 15), \dots, (n - 4, n - 3)\}.
\end{align*}
We note that $X_0^l(g),~X_0^u(g)$, and $X_0^d(g)$ are pairwise disjoint. Therefore, $X_4(g)=\emptyset.$

From Corollary~\ref{coroD'}~(2), $D_4(g)$ is a dominating set of $I(n)$ with cardinality $|D_4(g)| = |g_n^{-1}(4)|$. Since $H^{0,3}(n) \cap g^{-1}(4) = \{(1, 2), (3, 3)\}$, from Proposition 2.4 we obtain:
\[
|D_4(g)| = 2 + \sum_{i=1}^{(n-3)/5}\left|H^{(5i+3)-4,5i+3}(n)\cap g^{-1}(4)\right| = 2 + \sum_{i=1}^{(n-3)/5}(5i+3) = \frac{n^2 + 5n - 4}{10}.
\]
Let the exchange sets near the corners (see Figure~\ref{fcason3}) be:
\begin{align*}
A_3 &:= \{ (1, 7), (1, 9), (2, 5), (1, n - 1), (1, n - 4), (3, n - 1), (n - 1, n - 1) \}, \\
B_3 &:= \{ (1, 8), (2, 6), (2, n - 1), (1, n - 3) \}.
\end{align*}

Set $D^*_4(g) := (D_4(g) \setminus A_3) \cup B_3$. Since $D_4(g)$ is a dominating set of $I(n)$, it suffices to verify that the closed neighborhood $N[A_3]$ remains dominated in $D^*_4(g)$. Based on the structure of $I(n)$, an explicit calculation yields the following:
\begin{align*}
N[A_3] = \{ 
  & \mathbf{(1, 7)}, (1, 6), (1, 8), (2, 7),  \mathbf{(1, 9)}, (1, 10), (2, 9),  \mathbf{(2, 5)}, (1, 5), (3, 5), (2, 4), (2, 6), \\
  & \mathbf{(1, n - 1)}, (1, n - 2), (2, n - 1), \mathbf{(1, n - 4)}, (1, n - 5), (1, n - 3), (2, n - 4), \\
  & \mathbf{(3, n - 1)}, (4, n - 1), (3, n - 2), \mathbf{(n - 1, n - 1)}, (n - 2, n - 1)\},
\end{align*}
whereas
\begin{align*}
N[B_3] = \{ 
  & \mathbf{(1, 8)}, (1, 7), (1, 9), (2, 8), 
 \mathbf{(2, 6)}, (1, 6), (3, 6), (2, 5), (2, 7), 
 \mathbf{(2, n - 1)}, (1, n - 1), \\
  &(3, n - 1), (2, n - 2),  \mathbf{(1, n - 3)}, (1, n - 4), (1, n - 2), (2, n - 3) \}.
\end{align*}

It remains to verify that $B_3$ dominates every vertex in $S_3:= N[A_3] \setminus N[B_3]$. It is easy to see that
\begin{align*}
S_3=\{& (1, 5), (1, 10), (1, n - 5), (2, 4), (2, 9), (2, n - 4), (3, 5), (3, n - 2), (4, n - 1),\\
       &  (n - 2, n - 1), (n - 1, n - 1) \}.
\end{align*}

The following table exhibits explicit dominators in $D_4(g) \setminus A_3$ for each vertex in $S_3$:
\begin{center}
\begin{tabular}{|c|c|}
\hline
$(x,y)\in D^*_4(g) \setminus A_3$ & Vertices of $S_3$ dominanted by $(x,y)$  \\
\hline
$(1,4)$& $(1,5), (2,4)$ \\
\hline
$(4,5) $& $ (3,5)$ \\
\hline
$(2,10)$&$ (1,10), (2,9)$ \\
\hline
$(1, n - 6)$ & $(1,n-5)$\\
\hline
$(2,n-3)$ & $(2,n-4)$\\
\hline
$(4,n-2)$ & $(3,n-2), (4,n-1)$\\
\hline
$(n-2,n-1)$ & $(n-2,n-1), (n-1,n-1)$\\
\hline
\end{tabular}
\end{center}

Therefore, $D^*_4(g)$ is a dominating set of $I(n)$. Since $|A_3| = 7$ and $|B_3| = 4$, then
\[ |D^*_4(g)| = \frac{n^2 + 5n - 4}{10} - |A_3| + |B_3| = \frac{n^2 + 5n - 4}{10} - 7 + 4 = \frac{n^2 + 5n - 34}{10}. \]
\begin{figure}[h]
\begin{center}
\includegraphics[scale=1.3]{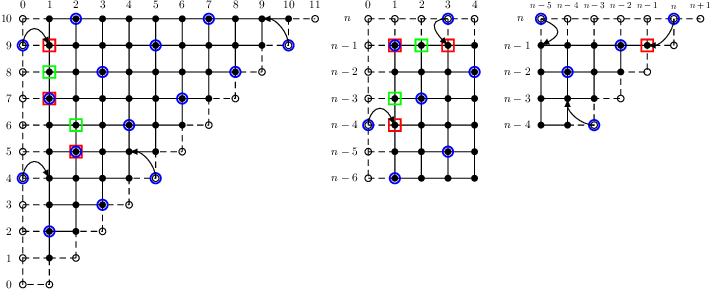}
\caption{The vertices of $A_3$ are enclosed in red squares, while those of $B_3$ are enclosed in green squares. We note that  $N[D_4(g)]\subseteq N[(D_4(g)\setminus A_3)\cup B_3]$.}
\label{fcason3}
\end{center}
\end{figure}

\vskip 0.3cm
\noindent\textbf{Case 4}. Suppose that $n \equiv 4 \pmod 5$. We take $h = g$ and $r = 0$. From Proposition~\ref{propgf}~(3)--(5) and the shift conditions of the auxiliary boundary sets, we obtain the following:
\begin{align*}
X_0^l(g) &= \{(1, 1), (1, 5), (1, 10), \dots, (1, n - 4)\}, \\
X_0^u(g) &= \{(3, n - 1), (8, n - 1), (13, n - 1), \dots, (n-1, n-1)\}, \\
X_0^d(g) &= \{(1, 2), (6, 7), (11, 12), \dots, (n - 3, n - 2)\}.
\end{align*}
We note that $X_0^l(g),~X_0^u(g)$, and $X_0^d(g)$ are pairwise disjoint. Therefore, $X_4(g)=\emptyset.$

From Corollary~\ref{coroD'}~(2), $D_0(g)$ is a dominating set of $I(n)$ with cardinality $|D_0(g)| = |g_n^{-1}(0)|$. Since $H^{0,4}(n)\cap g^{-1}(0)=\{(0, 0), (2,1),(1,3),(3,4)\}$ contributes exactly $4$ elements near the initial corner, from Proposition 2.4 we obtain:
\[
|D_0(g)| = 4 + \sum_{i=1}^{(n-4)/5}\left|H^{(5i+4)-4,5i+4}(n)\cap g^{-1}(0)\right| = 4 + \sum_{i=1}^{(n-4)/5}(5i+4) = \frac{n^2 + 5n + 4}{10}.
\]
Let the exchange sets near the corners (see Figure~\ref{fcason4}) be:
\begin{align*}
A_4 &:= \{ (1, 1), (1, 3), (1, 5), (1, 8), (1, 13), (1, 15), (2, 6), (2, 11), (3, 9), (4, 7), (6, 8), \\
    &\phantom{:= \{} (1, n - 4), (1, n - 1), (3, n - 1), (n - 6, n - 1), (n - 3, n - 1) \}, \\
B_4 &:= \{ (1, 6), (1, 14), (2, 4), (2, 8), (2, 12), (3, 6), (3, 10), (4, 8), (7, 9), \\
    &\phantom{:= \{} (1, n - 3), (2, n - 1), (n - 5, n - 1) \}.
\end{align*}

Set $D^*_0(g) := (D_0(g) \setminus A_4) \cup B_4$. Since $D_0(g)$ is a dominating set of $I(n)$, it suffices to verify that the closed neighborhood $N[A_4]$ remains dominated in $D^*_0(g)$. Based on the structure of $I(n)$, an explicit calculation of the neighborhoods gives the following:
\begin{align*}
N[A_4] = \{ 
  & \mathbf{(1, 1)}, (1, 2), (2, 1), 
   \mathbf{(1, 3)}, (1, 4), (2, 3), 
   \mathbf{(1, 5)}, (1, 6), (2, 5), \mathbf{(1, 8)}, (1, 7), (1, 9), (2, 8), \\
  &  \mathbf{(1, 13)}, (1, 12), (1, 14), (2, 13), 
   \mathbf{(1, 15)}, (1, 16), (2, 15), 
  \mathbf{(2, 6)}, (3, 6), (2, 7), \mathbf{(2, 11)}, \\
  &  (1, 11), (3, 11), (2, 12), (2, 10), 
  \mathbf{(3, 9)}, (2, 9), (3, 8), (3, 10), (4, 9), 
 \mathbf{(4, 7)}, (3, 7), (4, 6), \\
  & (4, 8), (5, 7),  \mathbf{(6, 8)}, (5, 8), (6, 7), (6, 9), (7, 8), 
   \mathbf{(1, n - 4)}, (1, n - 5), (1, n - 3),  \\
  & (2, n - 4), \mathbf{(1, n - 1)}, (1, n - 2), (2, n - 1), 
   \mathbf{(3, n - 1)}, (4, n - 1),  \\
  &(3, n - 2), \mathbf{(n - 6, n - 1)}, (n - 7, n - 1), (n - 5, n - 1), (n - 6, n - 2),  \\
  &\mathbf{(n - 3, n - 1)}, (n - 4, n - 1), 
   (n - 2, n - 1), (n - 3, n - 2) \},
\end{align*}
whereas
\begin{align*}
N[B_4] = \{ 
  & \mathbf{(1, 6)}, (1, 5), (1, 7), (2, 6), \mathbf{(1, 14)}, (1, 13), (1, 15), (2, 14), 
 \mathbf{(2, 4)}, (1, 4), (3, 4), (2, 3),  \\
  & (2, 5), \mathbf{(2, 8)}, (1, 8), (3, 8), (2, 7), (2, 9), 
  \mathbf{(2, 12)}, (1, 12), (3, 12), (2, 11), (2, 13), 
   \mathbf{(3, 6)}, \\
  &(4, 6), (3, 5), (3, 7),  \mathbf{(3, 10)}, (2, 10), (4, 10), (3, 9), (3, 11), 
   \mathbf{(4, 8)}, (5, 8), (4, 7), (4, 9), \\
  & \mathbf{(7, 9)}, (6, 9), (8, 9), (7, 8), (7, 10),  \mathbf{(1, n - 3)}, (1, n - 4), (1, n - 2), (2, n - 3), \\
  & \mathbf{(2, n - 1)}, (1, n - 1), (3, n - 1), (2, n - 2), 
  \mathbf{(n - 5, n - 1)}, (n - 6, n - 1), \\
  &(n - 4, n - 1), (n - 5, n - 2) \}.
\end{align*}

It remains to verify that $B_3$ dominates every vertex in $S_4:= N[A_4] \setminus N[B_4]$. It is easy to see that
\begin{align*}
S_4=\{ & (1, 1), (1, 2), (1, 3), (1, 9), (1, 11), (1, 16), (1, n - 5), (2, 1), (2, 15),\\  
       &(2, n - 4),  (3, n - 2), (4, n - 1), (5, 7), (6, 7), (6, 8), (n - 7, n - 1),\\ 
       & (n - 6, n - 2), (n - 3, n - 2), (n - 3, n - 1), (n - 2, n - 1) \}.
\end{align*}

The following table exhibits explicit dominators in $D_0(g) \setminus A_4$ for each vertex in $S_4$:
\begin{center}
\begin{tabular}{|c|c|}
\hline
$(x,y)\in D^*_0(g) \setminus A_4$ & Vertices of $S_4$ dominanted by $(x,y)$  \\
\hline
$(1,2)$& $(1,1), (1,2), (1,3), (2,1)$ \\
\hline
$(1,10)$& $(1,9), (1,11)$ \\
\hline
$(2,16)$& $(1,16), (2,15)$ \\
\hline
$(6,7)$& $(5,7), (6,7),(6,8)$ \\
\hline
$(1,n-6)$& $(1,n-5)$ \\
\hline
$(2,n-3)$& $(2,n-4)$ \\
\hline
$(n-8,n-1)$& $(n-7,n-1)$ \\
\hline
$(n-5,n-2)$& $(n-6,n-2)$ \\
\hline
$(4,n-2)$& $(3,n-2), (4,n-1)$ \\
\hline
$(n-3,n-2)$& $(n-3,n-2), (n-3,n-1)$ \\
\hline
$(n-1,n-1)$& $(n-2,n-1)$ \\
\hline
\end{tabular}
\end{center}
Therefore, $D^*_0(g)$ is a valid dominating set of $I(n)$. Since $|A_4| = 16$ and $|B_4| = 12$, then
\[ |D^*_0(g)| = \frac{n^2 + 5n + 4}{10} - |A_4| + |B_4| = \frac{n^2 + 5n + 4}{10} - 16 + 12 = \frac{n^2 + 5n - 36}{10}. \]

\begin{figure}[H]
\begin{center}
\includegraphics[scale=1.0]{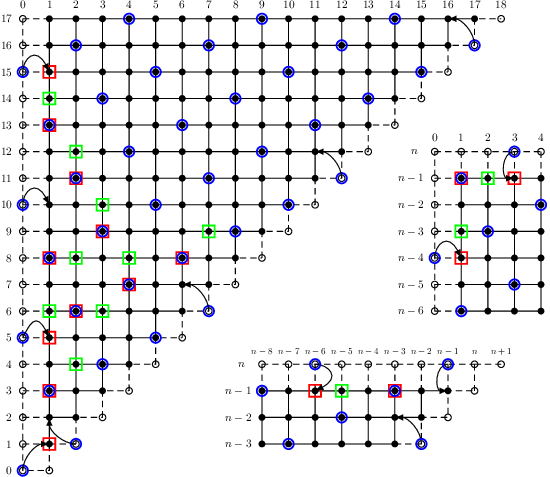}\\
\caption{The vertices of $A_4$ are enclosed in red squares, while those of $B_4$ are enclosed in green squares. We note that  $N[D_0(g)]\subseteq N[(D_0(g)\setminus A_4)\cup B_4]$.}
\label{fcason4} 
\end{center}
\end{figure}

\section{The algorithms}\label{s:pseudocode} To determine the exact value of $\gamma(F_2(P_n))$ for $n\in \{2,3,\ldots ,25\}$, 
we employed a computational approach based on the dynamic programming framework for grid graphs proposed by Alanko et al. \cite{alanko}. The {\em triangular grid graph} $G_{tri}(n)$ is the subgraph of $P_{n} \square P_{n}$ induced by the vertex set $V_{tri}:=\{v_{i,j} \mid 1 \le i \le n, 1 \le j \le i\}$, which is isomorphic to $F_2(P_{n+1})$, as noted in Section~\ref{prelimi}.  Vertices are ordered lexicographically: $v_{i,j} < v_{k,l}$ if $i < k$, or if $i = k$ and $j < l$. For any subset $V' \subseteq V_{tri}$, let 
$D(V') := \bigcup_{v \in V'} N[v]$ denote the set of vertices dominated by $V'$. For a state $S \subseteq V_{tri}$, let $s(S)$ denote the lexicographically smallest vertex in $V_{tri} \setminus S$ that is not dominanted by $S$.  

The main algorithm employs a Breadth-First Search (BFS) combined with dynamic programming to find $\gamma_{tri}(n):=\gamma(G_{tri})$ and is based on the following:

\begin{enumerate}
    \item \textbf{Search process}: The algorithm maintains a collection $\mathcal{S}$ of sets of dominated vertices, beginning with the empty set. For each $S \in \mathcal{S}$, it evaluates all vertices that could potentially dominate $s(S)$.
    \item \textbf{Pruning criteria}: To optimize the search, the algorithm rejects states using the following:
    \begin{itemize}
        \item \textbf{Theorem 2.2}~\cite{alanko}: If $|V_1| = |V_2|$ and $D(V_1) \subseteq D(V_2)$, then $V_1$ can be ignored.        \item \textbf{Theorems 2.3-2.5}~\cite{alanko}: Specific rules allow for the exclusion of candidate vertices based on the local configuration of $S$ around $s(S)$.
    \end{itemize}
    \item \textbf{Termination}: The search concludes when the entire grid is dominated. The depth of the BFS level at this point represents the domination number $\gamma_{tri}(n)$ (Theorem 2.1~\cite{alanko}).
\end{enumerate}

{\bf Algorithm 1} computes the domination number of a triangular grid of size $n$ by traversing the graph and maintaining a set of feasible states $\mathcal{S}$. It relies on three primary subroutines: \texttt{IsCandidateValid}, which incorporates Theorems 2.3-2.5 to discard redundant candidates; \texttt{IsPrunable}, which implements the pruning criteria of Theorem 2.2; and \texttt{IsFeasible}, which identifies and discards states containing ``forgotten" (undominated) vertices. These procedures are detailed in {\bf Algorithm 2}.

The core of the proposed algorithms relies on the iterative expansion of the state space $\mathcal{S}$. In the main loop, for each state $S \in \mathcal{S}$ and each valid candidate $v_c$ (\texttt{IsCandidateValid}) that dominates the current target vertex $v_{current}$, we generate a new state $S_{next}$ defined by:
\begin{equation}
    S_{next} = S \cup N[v_c].
\end{equation}
This operation represents the evolution of the solution space. The union operation ensures that the coverage information is propagated across the sweep-line process.

The efficiency of this algorithm depends on the strict validation of $S_{next}$ before its inclusion in the state set $\mathcal{S}_{new}$. If $S_{next}$ fails to satisfy the pruning criteria (\texttt{IsPrunable}) or the feasibility condition (\texttt{IsFeasible}), the state is discarded. This pruning process is critical for two reasons:

\begin{itemize}
    \item \textbf{Redundancy elimination (\texttt{IsPrunable}):} This filter discards states that are dominated by existing configurations. If a new state $S_{next}$ offers no better coverage than a previously stored state $S_{exist}$ with a higher or equal number of dominators, $S_{next}$ is identified as redundant and pruned.
    \item \textbf{Feasibility guarantee (\texttt{IsFeasible}):} This mechanism enforces the constraints of the sweep-line approach. If $S_{next}$ leaves a ``forgotten" vertex (a vertex that the sweep-line has already passed) undominated, it is impossible to reach a valid domination set from this branch. Consequently, the branch is cut, preventing the algorithm from exploring invalid solution paths.
\end{itemize}

By discarding these states, the algorithm prevents the exponential growth of the state space, maintaining only the most promising and valid configurations in each iteration.

\begin{algorithm}
\caption{Domination Number for a Triangular Grid}
\begin{algorithmic}[1]
\State \textbf{Input:} Size $n$ of the triangular grid $G_{tri}$
\State \textbf{Output:} Domination number $\gamma_{tri}(n)$
\State $\mathcal{S} \leftarrow \{\emptyset\}$ 
\For{$i = 1$ to $n$}
    \For{$j = 1$ to $i$}
        \State $\mathcal{S}_{new} \leftarrow \emptyset$
        \ForAll{$S \in \mathcal{S}$}
            \State $v_{current} = s(S)$ \Comment{Smallest non-dominated vertex}
            \ForAll{candidates $v_c$ that dominate $v_{current}$} 
                \If{$v_c \in V_{tri}$ \textbf{AND} \Call{IsCandidateValid}{$v_c, S, v_{current}$}}
                    \State $S_{next} = S \cup N[v_c]$ 
                    \If{\textbf{NOT} \Call{IsPrunable}{$S_{next}, \mathcal{S}$} \textbf{AND} \Call{IsFeasible}{$S_{next}$}}
                        \State Add $S_{next}$ to $\mathcal{S}_{new}$
                    \EndIf
                \EndIf
            \EndFor
        \EndFor
        \State $\mathcal{S} \leftarrow \mathcal{S}_{new}$
    \EndFor
\EndFor
\State \textbf{return} minimum size in $\mathcal{S}$
\end{algorithmic}
\end{algorithm}
\begin{algorithm}
\caption{Validation and Pruning Logic}
\begin{algorithmic}[1]
\Function{IsCandidateValid}{$v_c, S, v_{current}$}
       \State Let $(i, j)$ be the coordinates of $v_{current}$ \Comment{Extract coordinates from the vertex}
    \If{$v_c == v_{i-1,j}$ OR $v_c == v_{i,j-1}$} \Comment{Application of Theorem 2.3}
        \State \textbf{return} \textbf{False} 
    \EndIf
    
    \If{$v_{i,j+1} \in S$ AND $j\leq n-1$ AND $v_c == v_{i,j}$} \Comment{Application of Theorem 2.4}
        \State \textbf{return} \textbf{False} 
    \EndIf
    
    \If{$v_{i,j+1} \in S$ AND $v_{i,j+2} \in S$ AND $i \le m-1$, $j\leq n-2$ AND $v_c == v_{i,j+1}$} 
        \State \textbf{return} \textbf{False}\Comment{Application of Theorem 2.5}
    \EndIf
    
    \State \textbf{return} \textbf{True}
\EndFunction
\Function{IsPrunable}{$S_{p}, \mathcal{S}$}
    \ForAll{$S_{exist} \in \mathcal{S}$}
        \If{$D(S_{p}) \subseteq D(S_{exist})$ \textbf{AND} $|S_{p}| \ge |S_{exist}|$}
            \State \textbf{return} \textbf{True}
        \EndIf
    \EndFor
    \State \textbf{return} \textbf{False}
\EndFunction

\Function{IsFeasible}{$S_{p}$}
    \ForAll{$v \in V_{tri}$ that can no longer be dominated}
        \If{$v \notin D(S_{p})$}
            \State \textbf{return} \textbf{False}
        \EndIf
    \EndFor
    \State \textbf{return} \textbf{True}
\EndFunction
\end{algorithmic}
\end{algorithm}

\section*{Acknowledgments}
E. Acosta Troncoso is a Phd student of the Programa de Doctorado en Ciencias B\'asicas at Universidad Aut\'onoma de Zacatecas (UAZ) and has received a fellowship (Grant No. 4014068) from Secretar\'ia de Ciencia, Humanidades, Tecnolog\'ia e Innovaci\'on (SECIHTI), formerly Consejo Nacional de Humanidades, Ciencias y Tecnolog\'ias (CONAHCYT).\\

We also thank the two anonymous referees for their helpful suggestions and recommendations, which significantly improved the presentation of this paper.

\section*{Declarations}

Not applicable.

\end{document}